\documentclass[11pt]{article}

\usepackage[top=50mm, bottom=50mm, left=50mm, right=50mm]{geometry}
\usepackage{lineno}
\usepackage{ stmaryrd }
\usepackage{amssymb}
\usepackage{amsmath}
\usepackage{amsthm}
\usepackage{epsfig}
\usepackage{graphicx}
\usepackage{graphics}
\usepackage{float}
\usepackage{subfigure}
\usepackage{multirow}
\usepackage{color}
\usepackage{lineno}
\usepackage{fullpage}
\usepackage[normalem]{ulem} 
\usepackage{makeidx}
\usepackage{xspace}
\usepackage{wrapfig}
\usepackage{tikz}
\usetikzlibrary{arrows,decorations.pathreplacing}
\makeindex

\newtheorem{theorem}{Theorem}

\newtheorem{corollary}{Corollary}
\newtheorem{lemma}{Lemma}

\newtheorem{proposition}{Proposition}

\newtheorem{question}{Question}
\newtheorem{manualtheoreminner}{Theorem}
\newenvironment{manualtheorem}[1]{%
  \renewcommand\themanualtheoreminner{#1}%
  \manualtheoreminner
}{\endmanualtheoreminner}

\newcommand{\A}{\mathbb A}
\newcommand{\nats}{\mathbb N}

\DeclareMathOperator{\cl}{Cl}
\DeclareMathOperator{\op}{Op}
\DeclareMathOperator{\Fact}{Fact}
\DeclareMathOperator{\Factr}{RecFact}

\begin{document}

\title{Open and closed complexity of infinite words}
 
\author{Olga Parshina\textsuperscript{$\triangleright$}\thanks{Supported by the Ministry of Education, Youth and Sports of the Czech Republic within the project no.~CZ.02.1.01/0.0/0.0/16019/0000778 and by the Ministry of Science and Higher Education of the Russian Federation (agreement no. 075–15–2022–287).} \and {Micka\"el Postic\textsuperscript{$\triangleleft$}
}}

\date{}

\maketitle
\vspace{-0.8cm}
\begin{center}
\textsuperscript{$\triangleright$} \small{Saint Petersburg University, 7/9 Universitetskaya nab., St. Petersburg, 199034 Russia \\
\textsuperscript{$\triangleleft$} Institut Camille Jordan, Universit\'e Claude Bernard Lyon 1, 43 bd du 11 novembre 1918,\\ Villeurbanne, 69622 France}
\end{center}
\bigskip

\setcounter{page}{1}

\begin{abstract}
In this paper we study the asymptotic behaviour of two relatively new complexity functions defined on infinite words and their relationship to periodicity. Given a factor $u$ of an infinite word $x$, we say $u$ is closed if it is a letter or if it is a complete first return to some factor $v$ of $x$; otherwise $u$ is said to be open. We show that for an aperiodic word $x$  over a finite alphabet, the complexity functions that count the number of closed and the number of open factors of $x$ of each given length are both unbounded. More precisely, we show that if $x$ is aperiodic then the limit inferior of the function of open complexity is infinite, and the limit superior of the function of closed complexity is infinite on any syndetic subset of positive integers. On the other hand, there exist aperiodic words for which limit inferior of the closed complexity function is finite.

\end{abstract}

\section{Introduction}

A fundamental problem in many areas of mathematics is to describe local constraints that imply global regularities. An example of this local to global phenomena is found in the study of periodicity in the framework of symbolic dynamics. 
The factor complexity function $p_x$, first introduced by G.A.~Hedlund and M.~Morse in their 1938 seminal paper on Symbolic Dynamics \cite{MHed}, counts the number of distinct blocks (or {\it factors}) of each length occurring in an infinite word $x=x_1x_2x_3\cdots  $ over a finite set $\A$.
They proved that each {\it aperiodic} (meaning not ultimately periodic) infinite word contains at least $n+1$ distinct factors of each length $n,$ and hence in particular the sequence $(p_x(n))_{n\in \nats}$ is unbounded. They further showed that an infinite word $x\in \A^\nats$ has exactly $n+1$ distinct factors of each length $n$ if and only if $x$ is binary, aperiodic and balanced, i.e., $x$ is a Sturmian word (see \cite{MH2}). 
Sturmian words are aperiodic words of lowest factor complexity and they arise naturally in different areas of mathematics including combinatorics, algebra, number theory, ergodic theory, dynamical systems and differential equations. 

There are numerous variations and extensions of the Morse-Hedlund theorem associated with other complexity functions defined on infinite words $x\in \A^\nats$ including  {\it Abelian complexity}~\cite{CovHed, RSZ}, which counts the number of distinct Abelian classes of words of each length occurring in $x$, or  {\it palindrome complexity}~\cite{ABCD} counting the number of distinct palindromes of each length occurring in $x$, or {\it cyclic complexity} \cite{CFSZ} counting the number of conjugacy classes of factors of each length in~$x$.
As in the case of the Morse and Hedlund theorem, in most cases these different complexity functions may be used to characterise aperiodicity in words.    

In this paper we investigate two new and complementary complexity functions defined on infinite words, and their relation to aperiodicity. 
Given a non-empty set $\A,$ called the {\it alphabet}, we let $\A^+$ denote the free semigroup generated by $\A$, and $\A^\nats$ the set of all right infinite words  $x~=~x_1x_2x_3\cdots$ with each $x_i \in \A.$ We will in general assume the alphabet to be finite unless stated otherwise. 
Following terminology first introduced by G.~Fici in~\cite{Fici_oc}, we say a word $w\in \A^+$ is {\it closed} if either $w\in \A$ or there exists $v\in \A^+$ which occurs in $w$ precisely twice, once as a prefix and once as a suffix.
Otherwise we say $w$ is {\it open}. 
For example, $w=abaaaab$ is closed (taking $v$ as above equal to $ab$) while $aabab$ and $aabaaa$ are both open.

We consider the complexity functions $\cl_x$ (resp.\ $\op_x)$ which associate to each infinite word $x\in \A^\nats$ the number of closed (resp.\ open) factors of $x$ of each given length. 
We note that every closed word $w\in \A^+$ either belongs to $\A$ or may be written in the form $w=uv=vu'$ for some choice of $u, u', v\in \A^+$, and moreover $w$ has no other occurrences of $v$ other than the two witnessed by the above factorisations. 
Thus in the language of symbolic dynamics, a closed factor $w\in \A^+\setminus \A$ of an infinite word $x\in \A^\nats$ is called a {\it complete first return to $v$ in $x$} and the factor $u$ is called a {\it return word} or {\it first return} to $v$ in $x.$ 

Return words constitute a powerful tool in the study of symbolic dynamical systems. For example, they play an important role in the theory of substitution dynamical systems. Return words were used by F.~Durand~\cite{Durand} and independently by C.~Holton and L.Q.~Zamboni in~\cite{HolZam1} to define so-called {\it derived words} and {\it derived substitutions} both of which may be used to characterise infinite words generated by primitive substitutions.
An analogous characterisation was later discovered by N.~Priebe~\cite{Priebe} in the framework of bi-dimensional tilings using the notion of derived tilings involving Vorono\"{\i} cells.
In~\cite{DurHS}, F.~Durand et al.\ derived a simple algorithm using return words for computing the dimension group of minimal Cantor systems arising from primitive substitutions.

A slightly different notion of return words was used by S.~Ferenczi, C.~Mauduit and A.~Nogueira~\cite{FerMaudNog} to compute the eigenvalues of the dynamical system associated to a primitive substitution.
Return words were an essential tool used by T.~Harju, J.~Vesti and L.Q.~Zamboni in~\cite{HVZ} to give a partial answer to a question posed by A.~Hof, O.~Knill and B.~Simon in~\cite{HKS} on a sufficient combinatorial criterion on the subshift $\Omega$ of the potential of a discrete Schr\"odinger operator which guarantees purely singular continuous spectrum on a generic subset of $\Omega.$ 

There are many other examples of the use of return words in the study of more general symbolic dynamical systems. 
In~\cite{Vuillon}, L.~Vuillon showed that an infinite binary word $x$ is Sturmian if and only if each factor of $x$ admits exactly two first returns in $x.$  
We observe that a recurrent word $x\in \A^\nats$ containing a factor $v$ having only one first return in $x$ is necessarily ultimately periodic, i.e., $x=u'u^\nats $ where $u$ is the unique first return to $v$ in $x.$
Words having exactly $k$ first returns to each factor for $k\geq 3$ have also been extensively studied (see for example \cite{BalPelST}) and include the symbolic coding of orbits under a $k$-interval exchange transformation~\cite{KS} as well as Arnoux-Rauzy words~\cite{AR} on a $k$-letter alphabet. Finally, there has been much recent interest in open and closed words in the framework of combinatorics on words and we refer the interested reader to the survey article by G.~Fici~\cite{surveyFici}.

Given an infinite word $x\in \A^\nats,$ we are interested in the asymptotic behaviour of the complexity functions $\cl_x$ and $\op_x$ and their relationship to periodicity. As every finite word $w\in \A^+$ is either open or closed, one has $p_x(n)=\op_x(n)+\cl_x(n)$ for each $n\in\nats$. 
Thus if $x$ is aperiodic, then it follows by the Morse and Hedlund theorem that at least one of the two sequences $(\op_x(n))_{n\in \nats},\, \, (\cl_x(n))_{n\in \nats}$ is unbounded. 
For instance, in~\cite{ParZamb} the first author together with L.Q.~Zamboni obtained explicit formulae for the closed and open complexity functions for Arnoux-Rauzy words on a $k$-letter alphabet (and hence in particular Sturmian words). 
They also showed that $\liminf \cl_x(n)=+\infty$ when $x$ is an Arnoux-Rauzy word.
However, for a general aperiodic word, the $\liminf \cl_x(n)$ may be finite, and in fact in~\cite{SchSh}, L.~Schaeffer and J.~Shallit proved that for the regular paperfolding word one has $\liminf \cl_x(n)=0,$ which is somewhat surprising. 
More generally, they showed that in the case of automatic sequences, the property of being closed is expressible in first-order logic, which allows them to compute the closed complexity for various well known infinite words including  the Thue-Morse word, the Rudin-Shapiro word, the ordinary paperfolding word and the period-doubling word (see for instance~\cite{AS}). 
One essential difference between the usual factor complexity on one hand, and the open and closed complexities on the other, is that the latter complexities are not in general monotone (e.g.\ see~\cite{ParZamb}).

The main result of this paper constitutes a refinement of the Morse-Hedlund theorem and may be stated as follows:

\begin{theorem}\label{th:main}
Let $x$ be an infinite word over a finite alphabet. The following are equivalent:
\begin{enumerate}
    \item $x$ is aperiodic;
     \item  $\liminf\limits_{n\rightarrow+\infty} \op_x(n) = +\infty$;
    \item $\limsup\limits_{n\rightarrow+\infty}\cl_x(n)= +\infty$.
\end{enumerate}
\end{theorem}

In particular, both complexity functions are unbounded if $x$ is aperiodic. Actually we prove something slightly more general in which condition 2. is replaced by $\limsup\limits_{n\in S}\cl_x(n)= +\infty$, where $S$ is any syndetic subset of $\nats.$ Of course, that conditions $2.$ and $3.$ each imply $1.$, is an immediate consequence of the Morse and Hedlund theorem. 
Since the limit inferior of the closed complexity of an aperiodic infinite word may be finite (as in the case of the regular paperfolding word) as it may be infinite (in the case of Sturmian words), we cannot hope to characterise periodicity in terms of $\liminf\limits \cl_x(n).$
Finally, it is necessary to assume the finiteness of the underlying alphabet, otherwise taking $x=1234567\cdots \in \nats^\nats$, we see that $x$ contains no closed factors of length greater than one. 

\medskip
We start with recalling some definitions and establishing some notations both of which will be pertinent in what follows. The implication $1.\Rightarrow 2.$ of Theorem~1 is proven in Section~\ref{sec:o}; 
in Section~\ref{sec:c} we show that words with bounded closed complexity are ultimately periodic, from which follows the implication $1.\Rightarrow 3.$
We finish with some remarks and open questions.

\section{Definitions and notations}

Let $\A$ be a finite non-empty set and $\A^+$ the free semigroup generated by $\A$ under the operation of concatenation of words. The set $\A$ is called the {\it alphabet} with its elements being {\it letters}. 
Given  $w=w_1w_2\cdots w_n \in \A^+,$ with each $w_i\in \A,$ the value $n$ is called the {\it length} of $w$ and is denoted with $|w|.$ We say $w\in \A^+$ is {\it primitive} if $w$ is not an integer power of some shorter word, i.e., if $w$ cannot be written in the form $w=u^n$ for some $u\in \A^+$ and integer $n\geq 2.$  
Given $u,w\in \A^+$ with $|u|<|w|,$ we say $u$ is a {\it border} of $w$ if $u$ is both a prefix and a suffix of $w.$ We say $w\in \A^+$ is {\it closed} if either $w\in \A$ or $w$ admits a border $u$ which occurs precisely twice in $w.$ Otherwise $w$ is said to be {\it open}. Thus $w\in \A^+$ is closed if either $w\in \A$ or if its longest border $u$ occurs exactly twice in $w,$ i.e., $u$ has no internal occurrences in $w.$ The longest border of a closed word is called \textit{frontier}. 

Let  $\A^\nats$ denote the set of all right infinite words $x=x_1x_2x_3\cdots$ with each $x_i\in \A.$ We endow $\A^\nats$ with the product topology of the discrete topology on $\A.$  For $x\in \A^\nats,$ we let $\Omega (x)$ denote the closure in $\A^\nats$ of the set $\{x_nx_{n+1}x_{n+2}\cdots \,|\, n\in \nats\}.$ $\Omega(x)$ is called the {\it shift orbit closure} of $x.$ 
An infinite word $x\in \A^\nats$ is said to be {\it purely periodic} if $x=u^\nats $ for some $u\in \A^+$ and  \textit{ultimately periodic} if $x=uv^{\nats }$ for some $u,v \in \A^+.$ We say $x$ is {\it aperiodic} if $x$ is not ultimately periodic. 

Given $x=x_1x_2x_3\cdots \in \A^+ \cup \A^\nats$ and $w\in \A^+,$ let $x| _w=\{m\in \nats\,|\, w=x_mx_{m+1}\cdots x_{m+|w|-1}\},$ i.e., $x| _w$ denotes the set of all occurrences of $w$ in $x.$ We say $w$ is a {\it factor} of $x$ if $x| _w\neq \emptyset.$ We say $w$ is a {\it recurrent factor} of $x$ if $x| _w$ is infinite. 
We let $\Fact(x)$ (resp.\ $\Factr(x))$ denote the set of factors (resp.\ recurrent factors) of $x.$ Thus $y\in \Omega(x)$ if and only if $\Fact(y)\subseteq \Fact(x).$ We say $x\in \A^\nats$ is recurrent if $\Fact(x)=\Factr(x).$ We say $x\in \A^\nats$ is {\it uniformly recurrent} if $x| _w$ is syndetic for every $w\in \Fact(x).$ Recall that a subset $S$ of $\nats$ is {\it syndetic} if there exists a positive integer $d$ such that $ S\cap \llbracket n,n+d\rrbracket \neq \emptyset$ for every $n\in \nats.$ It is a well known fact that $\Omega(x)$ contains at least one uniformly recurrent element. 

A factor $w$ of a finite or infinite word $x$ is called \textit{right special} (resp.\ \textit{left special}) in $x$ if there exist distinct letters $a$ and $b$ such that $wa$, $wb$ (resp.\ $aw$, $bw$) are factors of $x.$ For $x\in \A^\nats$ and $n\in \nats,$ let $\cl_x(n)$ (resp.\ $\op_x(n))$ denote the number of closed (resp.\ open) factors of $x$ of  length $n.$

\section{Words with finite $\liminf (\op_x(n))_{n\in \mathbb{N}}$ are ultimately periodic.}\label{sec:o}

In this section we make use of the properties of Rauzy graphs stated in Propositions~1 and~2.

For $x\in \A^\nats$ and $n\in \nats,$ the \textit{Rauzy graph} of order $n$ of $x\in \A^\nats$ is the directed graph whose set of vertices (resp.\ edges) consists of all factors of $x$ of length $n$ (resp.\ $n+1).$ There is a directed edge from $u$ to $v$ labeled $w$ if $u$ is a prefix of $w$ and $v$ a suffix of $w.$
A \textit{path} of length $k$ in a graph is an alternating sequence of vertices and edges $\ v_1,e_1,v_2,e_2,v_3,\dots,v_{k},e_{k},v_{k+1}\ $ which begins and ends with a vertex and where each $e_i$ is a directed edge from $v_i$ to $v_{i+1}.$ 
The \textit{distance} between two vertices in a Rauzy graph is the length of the shortest path between them.

The next two propositions as well as Corollary~\ref{rauzy2} also hold in case $\A$ is infinite. 
 
\begin{proposition}\label{l3}
Let $x\in \A^\nats$  and $N \in \mathbb{N}$. Let $w_1$ and $w_2$ be two factors of $x$, such that there is a path of length $i$ from $w_1$ to $w_2$ in the Rauzy graph of order $N$ of $x$. Suppose $w_1$ and $w_2$ are closed with frontiers $u_1$ and $u_2$ respectively. Then $||u_1|-|u_2||<i$.\\
In particular, if $i=1$ the frontiers are of the same length: $|u_1|=|u_2|$.

\end{proposition}

\begin{proof}
The situation is as illustrated on Figure~\ref{fig:prop8}. 


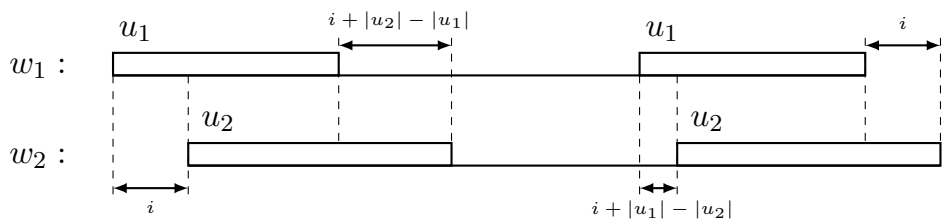
\begin{figure}[H]
\centering
\begin{tikzpicture}[thick,scale=1, every node/.style={scale=1.15}, >=latex] 
\centering
\pgfmathsetmacro\l{10} 
\pgfmathsetmacro\i{1}
\pgfmathsetmacro\u{3} 
\pgfmathsetmacro\v{3.5} 
\pgfmathsetmacro\h{1.2} 

\draw (0,0) -- (\l,0);
\draw (\i, -\h) -- (\l+\i,-\h);
\node at (-1,0.1) {$w_1:$};
\node at (-1,-\h+0.1) {$w_2:$};

\draw (0,0) rectangle (\u,0.3);
\node[above] at (0.3,0.3) {$u_1$};
\draw (\l-\u,0) rectangle (\l,0.3);
\node[above] at (\l-\u+0.3,0.3) {$u_1$};
\draw (\i,-\h) rectangle (\i+\v,0.3-\h);
\node[above] at (\i+0.4,0.3-\h) {$u_2$};
\draw (\l+\i-\v,-\h) rectangle (\l+\i,0.3-\h);
\node[above] at (\l+\i-\v+0.4,0.3-\h) {$u_2$};

\foreach \x in {0,\i,\l-\u,\l-\v+\i} 
   \draw[dashed, thin] (\x,0) -- (\x ,-\h-0.4);
\draw[<->] (0, -\h-0.3)--(\i, -\h-0.3);
\node[below] at (\i/2,-\h-0.3) {\tiny{$i$}};
\draw[<->] (\l-\u, -\h-0.3)--(\l-\v+\i, -\h-0.3);
\node[below] at (\l-\u+0.3,-\h-0.3) {\tiny{$i+|u_1|-|u_2|$}};
 
\foreach \x in {\u,\i+\v,\l,\l+\i}
    \draw[dashed, thin] (\x,0.5) -- (\x ,-\h+0.3     );  
\draw[<->] (\u, 0.4)--(\v+\i, 0.4);
\node[above] at (\u+0.8,0.45) {\tiny{$i+|u_2|-|u_1|$}};
\draw[<->] (\l, 0.4)--(\l+\i, 0.4);
\node[above] at (\l+\i/2,0.45) {\tiny{$i$}};

\end{tikzpicture}
\caption{Factors $w_1$ and $w_2$.}
    \label{fig:prop8}
\end{figure}

Since $w_2$ is closed, $u_2$ cannot be a factor of $u_1$. Hence $i+|u_2| -|u_1| > 0$. Since $w_1$ is closed, $u_1$ cannot be a factor of $u_2$. Hence $i+|u_1|-|u_2| > 0$. The result follows.
\end{proof}

\begin{corollary}\label{rauzy2}
Let $w_1,w_2,u_1,u_2$ be as in Proposition \ref{l3}. If there exists a path between $w_1$ and $w_2$ in the Rauzy graph consisting of only closed factors, then $|u_1|=|u_2|$. Thus, if there exists a path between $w_1$ and $w_2$ with $n$ distinct open factors, then $||u_1|-|u_2||\leq n$.
\end{corollary}

\begin{proposition}\label{5}
Let $x\in \A^\nats.$  For every $j>1$, every vertex in the Rauzy graph of order $j$ of $x$ has at most one closed predecessor and one closed successor.
\end{proposition}

\begin{proof}
%

Let $w$ be a word of length $j-1$ and consider $bw$ and $cw$ to be both closed with $b, c\in\A$, $b \neq c$. 
Then, labelling $u$ the frontier of $bw$ and $v$ the frontier of $cw$ one has $|u|\neq |v|$ since $u$ and $v$ are both suffixes of $w$ but do not start with the same letter. 
Suppose, without loss of generality, that $|u| < |v|$. This means $u$ is a proper suffix of $v$, hence it appears in $w$ as a proper suffix of the first occurrence of $v$ in $cw$. 
This leads to at least three occurrences of $u$ in $bw$, which is then not closed. Symmetrically, there is at most one letter $b'\in \A$ such that $wb'$ is closed.

\end{proof}

\begin{theorem}\label{open}
Let $x$ be an infinite word over a finite alphabet $\mathbb{A}$. Let $k$ be a positive integer such that\ \  $\liminf\limits_{n\in\nats} \op_x(n) = k$. Then $x$ is ultimately periodic.
\end{theorem}

In order to prove Theorem \ref{open}, we start by proving some lemmas, where $k$, $x$ and $\A$ are defined as in the theorem statement. 

\begin{lemma}\label{l1}
Suppose that $x$ is aperiodic. Let $N> 11k+2$ be such that $\op_x(N)=k$. Then $u^N \notin \Fact(x)$ for any choice of $u$ with $|u|<2k.$
\end{lemma}

\begin{proof}
Suppose that $u^N\in \Fact(x)$ for some primitive word  $u \in \A^+$ with $|u| \leq 2k.$
Since $x$ is aperiodic, up to considering a cyclic rotation of $u$ there exists $a \in \mathbb{A}$ such that $u^{\frac{N-1}{|u|}}a$ is a factor of $x$ with $u^{\frac{N-1}{|u|}}a\neq u^{\frac{N}{|u|}}$. 
This factor is open: if not, its frontier has length at least $N-1-|u|>3|u|+1$, which implies that $u$ occurs internally in $uu$ contradicting the fact that $u$ is primitive (see Figure~2). 
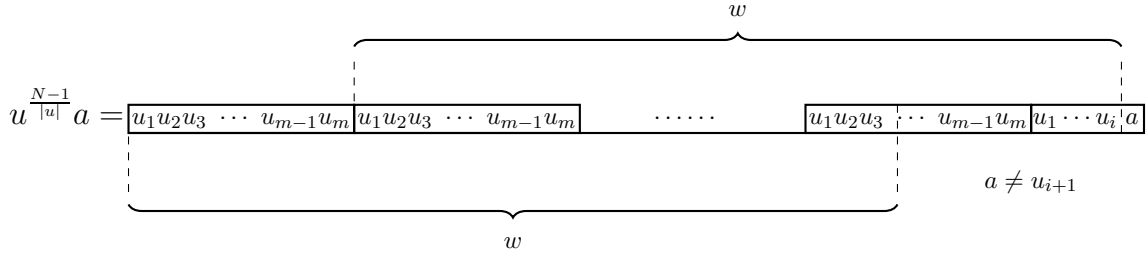
\begin{figure}[H]
\centering
\begin{tikzpicture}[thick,scale=0.75, every node/.style={scale=0.86}]
\centering
\pgfmathsetmacro\l{20}
\pgfmathsetmacro\i{1}
\pgfmathsetmacro\u{4}
\pgfmathsetmacro\v{3.5}

\node[above] at (-1,0) {\Large{$u^{\frac{N-1}{|u|}}a=\ $}};

\draw (0,0) -- (\l-4,0);

\foreach \x in {0,1,3}
    \draw (\x*\u,0) rectangle (\x*\u+\u,0.5);
\draw (4*\u,0) rectangle (4*\u+\u/2,0.5);
\foreach \x in {0,1,3}    
    \node[above] at (\x*\u+\u/2,-0.1) {$u_1u_2u_3\ \cdots\ u_{m-1} u_m$};    
\node[above] at (\l/2-0.1,0) {$\cdots\cdots$};
\node[above] at (4*\u+0.97,-0.07) {$u_1\cdots u_i\ a$};


\node[below] at (\l-4,-0.5) {$a\neq u_{i+1}$};

\draw[dashed, thin] (0,0.5) -- (0,-1.1);
\draw[dashed, thin] (\l-6.37,0.5) -- (\l-6.37,-1.1);
\draw [decorate,decoration={brace,raise=2mm,amplitude=5pt,mirror}] (0,-1) -- (\l-6.37,-1) node[below, pos=.5, yshift=-0.6cm] {$w$};
\draw[dashed, thin] (\u,0) -- (\u,1.4);
\draw[dashed, thin] (\u+\l-6.4,0) -- (\u+\l-6.4,1.4);
\draw [decorate,decoration={brace,raise=4mm,amplitude=5pt}] (\u,1) -- (\u+\l-6.4,1) node[above, pos=.5, yshift=0.8cm] {$w$};
\end{tikzpicture}
\caption{The frontier should be longer than $w$.}
    \label{fig:l3}
\end{figure}

Let us consider, for $j \leq k$, a factor $u_{j+1}\cdots u_{|u|}u^{\frac{N-1}{|u|}-1}ab_1\cdots b_j$, which is a successor of $u^{\frac{N-1}{|u|}}a$ at distance at most $k$ in the Rauzy graph of order $N$ of $x$. 
Again, this factor is open: otherwise the length of its frontier  would be at least $N-1-|u|-j>3|u|+1$, and $u$ would be a factor of $uu$. Besides, those factors are pairwise distinct, since the equality between any two of them would imply that $u$ is an internal factor of $uu$. This produces at least $k+1$  distinct open factors of length $N$, thereby contradicting our initial assumption on $N$.  
\end{proof}

\begin{lemma}\label{4}
Let $j\in \mathbb{N}$ be such that $\op_x(j)=k$. Let $u$ and $v$ be two closed factors of length $j$ whose frontiers are of length $r$ and $p$ respectively. Then $ |p-r|\leq k$. 
\end{lemma}
\begin{proof}
Consider the  Rauzy graph of $x$ of order $j.$ By Corollary~\ref{rauzy2}, it is enough to count the number of distinct open factors on a path between $u$ and $v$ to know the bound on $|p-r|$. There can be at most $k$ of them, so $|p-r|\leq k$.

\end{proof}
\begin{lemma}\label{l6}
Suppose $x$ is aperiodic. Let $m \in \mathbb{N}$, $t=|\A|$, and $N\geq k (t^m+m+2)$ such that $\op_x(N)=k$. Then the frontier of any closed factor of length $N$ is longer than $m$.

\end{lemma}
\begin{proof}
Since $x$ is aperiodic, it contains at least $N+1$ different factors of length $N$. 


By Proposition~\ref{5}, there exists a factor such that the shortest path in the Rauzy graph between it and an open factor is of length $\frac{N+1-k}{k}$. By Corollary~\ref{rauzy2}, all closed words on this path have frontiers of the same length.

Let us suppose that this common frontier length is smaller than $m$. There are at most $t^m <\frac{N+1-k}{k}$ such frontiers, so by the pigeon hole principle two of those factors have the same frontier with their distance in the Rauzy graph being less than $t^m+1$. 
Since this frontier cannot occur internally, the distance between those factors is at least $N-m$; hence $N-m<t^m+1$, contradicting the definition of $N$.
\end{proof}
\begin{proof}[Proof of Theorem~\ref{open}]
Let $m=(11k+3)k+2k$. In this case if a word of length at least $m-k$ overlaps itself with distance less than $k$, then it contains a power of exponent $11k+3$ with root shorter than $k$.
Let $N>k (t^m+m+2)$ be such that $\op_x(N)=k$. Consider a right special factor $w=w_1\cdots w_N$ of $x$ (which exists since $x$ is aperiodic). 
By Proposition \ref{5}, there exists $i \leq k$ such that $w_a=w_{i+1}\cdots w_Nay_1\cdots y_{i-1}$ and $w_b=w_{i+1}\cdots w_Nbz_1\cdots z_{i-1}$ with $a \neq b \in \mathbb{A}$ are both closed factors of $x$. See Figure \ref{frontiers1}: at each step before the rightmost one, either on top, bottom, or both paths, there must be an open factor, and each open factor can only appear once.

\begin{figure}[H]
    \centering
    \begin{tikzpicture}[thick,scale=0.6, every node/.style={scale=0.5}, >=latex]
    \usetikzlibrary{shapes}
\tikzstyle{every node}=[draw, rectangle, minimum height=15pt, align=center]
\node (0) at (0,0) {$w_1w_2w_3\cdots w_N$};
\node (1) at (6,1) {$w_2w_3\cdots w_N a$};
\node[above,draw=none] at (6,1.3) {\textit{open}};
\node (-1) at (6,-1) {$w_2w_3\cdots w_N b$};
\node[draw=none] (2) at (11,1) {$\cdots\cdots$};
\node[below,draw=none] at (6,-1.3) {\textit{closed}};
\node[draw=none] (-2) at (11,-1) {$\cdots\cdots$};
\node (i) at (18,1) {$w_{i+1}w_{i+2}\cdots w_N\ a\ y_1y_2\cdots y_{i-1}$};
\node[above,draw=none] at (18,1.3) {\textit{closed}};
\node (-i) at (18,-1) {$w_{i+1}w_{i+2}\cdots w_N\ b\ z_1z_2\cdots z_{i-1}$};
\node[below,draw=none] at (18,-1.3) {\textit{closed}};


\draw[thin,->] (0) -- (1);
\draw[thin,->] (0) -- (-1);
\draw[thin,->] (1) -- (2);
\draw[thin,->] (-1) -- (-2);
\draw[thin,->] (2) -- (i);
\draw[thin,->] (-2) -- (-i);
    \end{tikzpicture}
    
    \caption{The sequence of open and closed factors in the Rauzy graph of order $N$.}
    \label{frontiers1}
\end{figure}
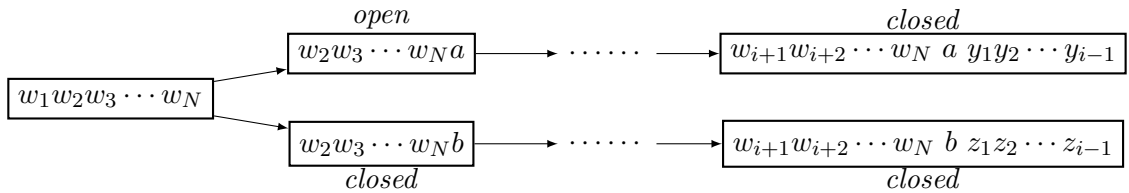

Let us denote the frontiers of $w_a$ and $w_b$ by $u$ and $v$ respectively. For the illustration of the following reasoning see Figure~\ref{fig:overlap}. Applying Lemma \ref{4}, we get $||u|-|v||\leq k$. Since both $u$ and $v$ are longer than $k$ and $a\neq b$, they cannot be equal. This implies $|u|\neq |v|$ since $w_a$ and $w_b$ have a long common prefix.  
Suppose, without loss of generality, that $|u|<|v|$. Lemma \ref{l6} gives $m<|v|$.
Let $u'$ and $v'$ be prefixes of $u$ and $v$ such that $u=u'ay_1\cdots y_{i-1}$ and $v=v'bz_1\cdots z_{i-1}$. Then, $|v'|>m-k$ and $u'$ is a prefix and a suffix of $v'$. Hence $v'$ overlaps itself with a difference less than $k$, what contradicts Lemma~\ref{l1}.


\begin{figure}[H]
\centering
\begin{tikzpicture}[thick,scale=1.1, every node/.style={scale=1.1}] 
\centering
\pgfmathsetmacro\l{10} 
\pgfmathsetmacro\i{2.7}
\pgfmathsetmacro\u{4} 
\pgfmathsetmacro\v{4.7} 
\pgfmathsetmacro\h{2} 

\draw (0,0) -- (\l,0);
\draw (0, -\h) -- (\l,-\h);
\node at (-2.5,0.1) {\footnotesize{$w_{i+1}w_{i+2}\cdots w_N\, a\, y_1y_2\cdots y_{i-1}:$}};
\node at (-2.5,-\h+0.1) {\footnotesize{$w_{i+1}w_{i+2}\cdots w_N\, b\, z_1z_2\cdots z_{i-1}:$}};

\node at (-2.5,-\h/2) {\footnotesize{$a\neq b$}};

\draw (0,0) rectangle (\u,0.4);
\node[above] at (0.4,-0.04) {\footnotesize{$u'$}};
\draw (0,0) rectangle (\u-\v+\i,0.4);
\draw (\l-\u,0) rectangle (\l,0.4);
\node[above] at (\l-\u+0.4,-0.04) {\footnotesize{$u'$}};
\draw (\l-\u,0) rectangle (\l-\v+\i,0.4);
\node[above] at (\l-\v+\i+1,-0.1) {\footnotesize{$a\,y_1y_2\cdots y_{i-1}$}};

\draw (0,-\h) rectangle (\v,0.4-\h);
\node[above] at (0.4,-\h-0.07) {\footnotesize{$v'$}};
\draw (0,-\h) rectangle (\i,0.4-\h);
\draw (\l-\v,-\h) rectangle (\l,0.4-\h);
\node[above] at (\l-\v+0.4,-\h-0.07) {\footnotesize{$v'$}};
\draw (\l-\v,-\h) rectangle (\l-\v+\i,0.4-\h);
\node[above] at (\l-\v+\i+1,-0.1-\h) {\footnotesize{$b\,z_1z_2\cdots z_{i-1}$}};

\draw [decorate,decoration={brace,raise=2mm,amplitude=5pt,mirror}] (0,-\h/2+0.6) -- (\l-\v+\i,-\h/2+0.6) node[below, pos=.5, yshift=-2.5mm] {\footnotesize{$w_{i+1}w_{i+2}\cdots w_N$}};
\draw [decorate,decoration={brace,raise=2mm,amplitude=5pt}] (0,-\h/2-0.5) -- (\l-\v+\i,-\h/2-0.5);

\draw [decorate,decoration={brace,raise=2mm,amplitude=5pt}] (0,0.45) -- (\u,0.45) node[above, pos=.5, yshift=3mm] {\footnotesize{$u$}};
\draw [decorate,decoration={brace,raise=2mm,amplitude=5pt}] (\l-\u,0.45) -- (\l,0.45) node[above, pos=.5, yshift=3mm] {\footnotesize{$u$}};

\draw [decorate,decoration={brace,raise=2mm,amplitude=5pt,mirror}] (\l-\v,-\h-0.1) -- (\l,-\h-0.1) node[below, pos=.5, yshift=-3mm] {\footnotesize{$v$}};
\draw [decorate,decoration={brace,raise=2mm,amplitude=5pt,mirror}] (0,-\h-0.1) -- (\v,-\h-0.1) node[below, pos=.5, yshift=-3mm] {\footnotesize{$v$}};

\draw[dashed, thin] (0,0) -- (0,-\h);
\draw[dashed, thin] (\l-\v+\i,0) -- (\l-\v+\i,-\h);

\draw[dashed, thin] (\l-\v,0) -- (\l-\v,0.4);
\draw[<->] (\l-\v, 0.15)--(\l-\u, 0.15);
\node[above, thin] at (\l-\v/2-\u/2,0.15) {\tiny{$\leq k$}};

\end{tikzpicture}
\caption{Factor $v'$ overlaps itself with difference smaller than $k$.}
    \label{fig:overlap}
\end{figure}
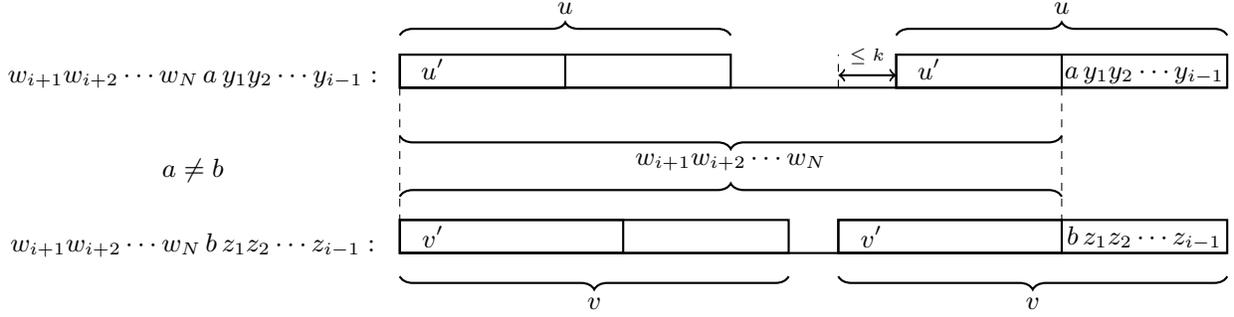

\end{proof}

\section{Words with bounded closed complexity are ultimately periodic}\label{sec:c}
The goal of this section is to prove the following theorem which gives a characterisation of ultimately periodic words in terms of closed complexity.


\begin{theorem}\label{close}
Let $x \in \mathbb{A}^{\mathbb{N}}$ be such that there exist a positive integer $d$ and a syndetic subset $S\subseteq \nats$ with gaps smaller than $d$ on which the closed complexity of $x$ is bounded, i.e.\ there exists $k \in \mathbb{N}$ such that $\cl_x(n) < k$ for every $n \in S$.
Then $x$ is ultimately periodic.
\end{theorem}
In what follows, $x$, $k$ and $S$ are defined as in the theorem.

The following lemma states that every recurrent factor is close to being right or left special.

\begin{lemma}{}\label{branch}
Let $x$ be aperiodic, and $u$ be a recurrent factor of $x$. 
Let $r$ and $s$ be words of lengths $k$ and $k+d$ respectively, such that $rus\in\Fact(x)$. Then there exist proper (probably, empty) suffix $r'$ of $r$ and prefix $s'$ of $s$, such that $r'us'$ is either right or left special in $x$.
\end{lemma}
\begin{proof}
Let us suppose the contrary, i.e.\ there exists a recurrent factor $u$ of $x$, such that $ru=r_k\cdots r_{1}u$ is its only recurrent left extension, and $us=us_1\cdots s_{k+d}$ is its only recurrent right extension. Up to considering a suffix $y$ of $x$, we can assume that every occurrence of $u$ is preceded by $r_k\cdots r_s$ and is followed by $s_1\cdots s_{k+d}$. 

Let us consider any complete first return to $rus$ denoting it with $w$. 
Note that $w$ does not contain any occurrences of $u$ but the two from its frontier $rus$: any extra occurrence of $u$ would add its unique extension $rus$ in $w$.

Let us take $d' \leq d$ such that the length of the prefix $w'$ of $w$ is $|w'|=|w|-k-d+d'\in S$.

We will now show that $w$ contains at least $k+1$ distinct closed factors of length $|w'|$, contradicting the assumption of Theorem~\ref{close}. Let us consider the set $C=\{c_i \}_{i=0}^k$ of factors of $w$, such that every $c_i$ has the border $r_{i}\cdots r_1us_1\cdots s_{d'}\cdots s_{d'+k-i}$ (see Figure~\ref{fig:kclos}). 
We claim that all words from $C$ are distinct and closed. Indeed, if a border of some factor from $C$ occurs in it internally, or if two factors from $C$ are equal, there is an extra occurrence of $u$ in $w$. Thus, there are at least $k+1$ distinct closed factors of $x$ of length $|w'|\in S$.
\end{proof}


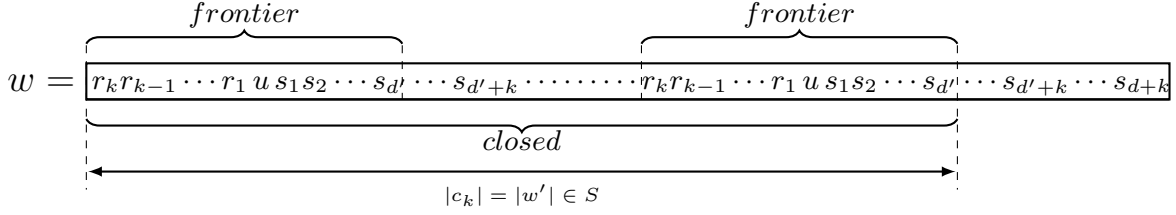
\begin{figure}[H]
\centering
\begin{tikzpicture}[thick,scale=1.2, every node/.style={scale=1.2}, >=latex] 
\centering
\pgfmathsetmacro\l{10} 
\pgfmathsetmacro\i{2.2}
\pgfmathsetmacro\u{3.5} 
\pgfmathsetmacro\h{1} 

\node[above] at (-0.5,-0.05) {\large{$w=$}};
\draw (0,0) -- (\l+2,0);

\draw (0,0) rectangle (\l+2,0.4);
\node[above] at (\l/2+1.04,-0.1) {\footnotesize{$\, r_kr_{k-1}\cdots r_1\, u\, s_1s_2\cdots s_{d'}\cdots s_{d'+k}\cdots\cdots\cdots r_kr_{k-1}\cdots r_1\, u\, s_1s_2\cdots s_{d'}\cdots s_{d'+k}\cdots s_{d+k}\, $}};

\draw [decorate,decoration={brace,raise=2mm,amplitude=5pt,mirror}] (0,-\h/2+0.45) -- (\l-0.35,-\h/2+0.45) node[below, pos=.5, yshift=-1.5mm] {\footnotesize{$closed$}};

\draw [decorate,decoration={brace,raise=2mm,amplitude=5pt}] (0,0.45) -- (\u,0.45) node[above, pos=.5, yshift=2mm] {\footnotesize{$frontier$}};
\draw [decorate,decoration={brace,raise=2mm,amplitude=5pt}] (\l-0.35-\u,0.45) -- (\l-0.35,0.45) node[above, pos=.5, yshift=2mm] {\footnotesize{$frontier$}};

\draw[dashed, thin] (0,0.6) -- (0,-\h);
\draw[dashed, thin] (\u,0.6) -- (\u,0);
\draw[dashed, thin] (\l-0.35-\u,0.6) -- (\l-0.35-\u,0);
\draw[dashed, thin] (\l-0.35,0.6) -- (\l-0.35,-\h);

\draw[<->] (0.02, -\h+0.2)--(\l-0.35-0.02, -\h+0.2);
\node[below] at (\l/2-0.175,-\h+0.2) {\tiny{$|c_k|=|w'|\in S$}};

\end{tikzpicture}
\caption{$k+1$ distinct closed factors of $x$.}
    \label{fig:kclos}
\end{figure}

\medskip
\begin{corollary}\label{unifrec}
If $x$ is uniformly recurrent, then it is periodic.
\end{corollary}

\begin{proof}
Let us suppose that $x$ is aperiodic (a uniformly recurrent word that is ultimately periodic is periodic). By Lemma~\ref{branch}, every factor is either close to being left special or is close to being right special. 
The idea of the proof is the following: using Lemma~\ref{branch}, we can produce a factor $u$ of $x$ such that it has a long first return. More precisely, if there exist words $w$ and $v$ such that $uw=vu$, then $|v|>2k+d$. Then using the same lemma we can construct an arbitrary long factor of $x$ that does not contain $u$, contradicting the uniform recurrence of $x$.

Let us begin with considering a recurrent factor $u$ of $x$. Without loss of generality we can assume it to be close to being right special.

Using Lemma~\ref{branch}, we can extend $u$ in a way that if there exist two factors $v$ and $w$ such that $uw=vu$, then $|v|>1$: at the first branching point $a_p$, where $u_1\cdots u_na_1\cdots a_p$ is a recurrent extension of $u$ and $p<k+d$, it is sufficient to take $a_{p+1} \neq a_{p}$ (or $a_1 \neq u_{n}$ if $u$ is right special). 
For the word $u^{(1)}=u_1\cdots u_na_1\cdots a_{p+1}$ we consider its extension $u^{(1)}_1\cdots u^{(1)}_{n+p+1}a'_1\cdots a'_{p'}$ to the right until the closest branching point if it is close to being right special, or its extension $a'_{p'}\cdots a'_1u^{(1)}_1\cdots u^{(1)}_{n+p+1}$ to the left if it is close to being left special. We denote the corresponding extension with $u^{(2)}$. 

To assure that if $u^{(2)}w'=v'u^{(2)}$ for some factors $v',w'$, then $|v'|>2$, we choose $a'_{p'+1}$ the following way. 
If $u^{(2)}$ is close to being right special and $p'< 2$, we take $a'_{p'+1}\neq u^{(1)}_{n+p+p'}$; if $u^{(2)}$ is close to being left special and $p'< 2$, we take $a'_{p'+1}\neq u^{(1)}_{2}$; otherwise $a'_{p'+1} \neq a'_{p'-1}$.

We apply recursively the same reasoning $2k+d+1$ times and obtain a recurrent factor $u^{(2k+d+1)}$ that for some factors $\bar w$ and $\bar v$ satisfies $\left(u^{(2k+d+1)}\bar w=\bar vu^{(2k+d+1)}\Rightarrow |\bar v|>2k+d+1\right)$. Let us note that the length of $u^{(2k+d+1)}$ is at least $2k+d+2$. 
For simplicity of notation, this factor is denoted by $u$ in the rest of the proof.

Since $x$ is uniformly recurrent, there exists $m \in \mathbb{N}$ such that every factor of $x$ of length $m$ contains $u$. 
Let us construct a factor contradicting this. 

We start with $u_{k+1}\cdots u_{|u|-(k+d)}$ and go to the next branching point given by Lemma \ref{branch}, that is either $u_{k+1}\cdots u_{|u|-(k+d)+p}$ for some $p<k+d$ or $u_{q}\cdots u_{|u|-(k+d)}$ for some $1<q\leq k$. 
At this point we choose a letter that differs from $u_{|u|-(k+d)+p+1}$ or from $u_{q-1}$. This ensures that $u$ does not occur before the next branching point.

\begin{figure}[H]
\centering
\begin{tikzpicture}[thick,scale=1.2, every node/.style={scale=1.2}, >=latex]
\usetikzlibrary{arrows}
\usetikzlibrary{shapes}
\centering
\pgfmathsetmacro\h{0.5}
\pgfmathsetmacro\u{5.7} 
\pgfmathsetmacro\i{1.5}

\node (0) at (-0.32,0) {$v_1v_2\,v_3\cdots\cdots\cdot\cdot v_j$};
\node (a) at (2.5,\h) {};
\node (b) at (2.5,-\h) {};
\node at (4.3,-\h) {$b\,b_2\,b_3\cdots\cdots\cdots\cdots b_{p'}$};
\node at (3.6,\h-0.1) {$a\,a_2a_3\cdots\cdot a_{p}$};
\node at (\u+2.3,0) {\footnotesize{$p'-p \leq 2k+d$}};

\draw[thin,->] (0) -- (a);
\draw[thin,->] (0) -- (b);

\draw [decorate,decoration={brace,raise=2mm,amplitude=5pt,mirror}] (0.4,-0.8) -- (0.4+\u,-0.8) node[below, pos=.5, yshift=-2.5mm] {$v$};
\draw [decorate,decoration={brace,raise=4mm,amplitude=5pt}] 
(0.5-\i,0.5) -- (0.5+\u-\i,0.5) node[above, pos=.5, yshift=4.5mm] {$u$};

\draw[dashed, thin] (0.4+\u,-0.8) -- (0.4+\u,0.7);
\draw[dashed, thin] (0.5+\u-\i,-0.2) -- (0.5+\u-\i,0.7);
\draw[<->] (0.55+\u-\i, 0.4) -- (0.35+\u, 0.4);
\node[above] at (\u-0.3,0.4) {\tiny{$p'-p$}};

\draw[dashed, thin] (0.4,-0.8) -- (0.4,-0.1);
\draw[dashed, thin] (0.5-\i,-0.6) -- (0.5-\i,0.7);
\draw[<->] (0.55-\i, -0.5) -- (0.35, -0.5);
\node[below] at (-0.3,-0.4) {\tiny{$p'-p$}};

\end{tikzpicture}
\caption{Equality $u = v$ would produce an overlap of $u$.}
    \label{fig:avoidu}
\end{figure}
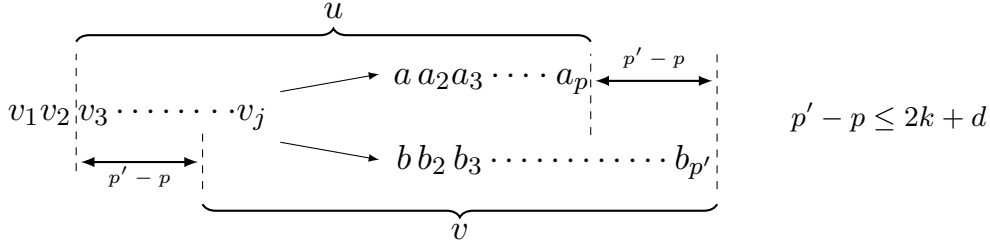


This reasoning can be applied to construct our factor recursively: at each branching point, knowing that $u$ does not appear before we can choose a branch such that $u$ will not occur after adding of any $2k+d+1$ letters to the right or to the left. Indeed, if it was not the case it would mean $u$ appears in both branches (see Figure~\ref{fig:avoidu} for the right special case). This allows us to construct, in at most $m$ steps, a factor longer than $m$ that does not contain $u$.
\end{proof}

The following lemma states that every periodic word in the subshift of $x$ has short period.

\begin{lemma}\label{period} Let $u$ be a primitive word in $\Fact(x)$ such that $u^n\in \Fact(x)$ for every $n\in \mathbb{N}$.\\ Then $|u|<k$.
\end{lemma}
\begin{proof}
Let $p<|u|$ and $n>2$ be such that $p+n|u|\in S$. 
Let us denote the $i$-th rotation of $u=u_1\cdots u_{|u|}$ by $r^i(u)=u_i\cdots u_{|u|} u_1\cdots u_{i-1}$. 
Since $u$ is primitive, so is $r^i(u)$, hence each $(r^i(u))^nu_i\cdots u_{i+p}$ is closed with frontier $(r^i(u))^{n-1}u_i\cdots u_{i+p}$. Indeed, if the frontier had an internal occurrence, then $r^i(u)$ would be an internal factor of  $(r^i(u))^2$, contradicting primitivity of $r^i(u)$. 
All the rotations $r^i(u)$ are pairwise distinct, and so are all closed factors $(r^i(u))^nu_i\cdots u_{i+p}$, $i=1,2,3,\dots,|u|$. By the assumption of Theorem~\ref{close}, $\cl_x(n) < k$ and thus $|u|<k$.
\end{proof}

\begin{proof}[Proof of Theorem$~\ref{close}$]
Let us suppose that $x$ is aperiodic and consider the set $P_x=\lbrace u \in \Fact(x) \,|\, \forall n \in \mathbb{N}, u^n\in \Fact(x)\rbrace$. 
By Lemma~\ref{period}, $P_x$ is finite. According to Lemma \ref{branch}, there exists $N\in \mathbb{N}$ such that we can produce an infinite word $y$ in the subshift  $\Omega(x)$ that does not contain $u^N$ as a factor for any $u \in P_x$. 
Moreover, since $y\in \Omega(x),$ it verifies Lemmas \ref{branch} and \ref{period}, and thus $P_y=\emptyset$. 
Let $z$ be any uniformly recurrent word in $\Omega(y)$. By Corollary~\ref{unifrec} the word $z$ is periodic and can be represented as $z=u^{\nats }$ for some finite word $u$. Then $u \in P_y$, what contradicts $P_y=\emptyset$. Hence $x$ is ultimately periodic. 
\end{proof}

\section{Concluding remarks and open questions}
 
Combining Theorem \ref{open}, Theorem \ref{close} and the fact that $\op_x(n)+\cl_x(n)=p_x(n)$ for every $n\in\nats$, we obtain the following result: 
 
\begin{manualtheorem}{\ref{th:main}$'$}
Let $x$ be an infinite word over a finite alphabet. The following are equivalent:
\begin{enumerate}
    \item $x$ is aperiodic;
    \item  $\liminf\limits_{n\rightarrow+\infty} \op_x(n) = +\infty$;
    \item $\forall S \subseteq \nats \text{ syndetic, }\limsup\limits_{n\in S} \cl_x(n) = +\infty$.
\end{enumerate}
\end{manualtheorem}

Since the factor complexity of an aperiodic word is a strictly increasing function,  $\liminf{p_x(n)}= + \infty$ is equivalent to $p_x$ being unbounded. However, it is not always the case for open and closed complexity functions (ex.\ see~\cite{ParZamb}).
Even though the result we obtained in terms of open complexity is as strong as Morse-Hedlund theorem since it is expressed in terms of $\liminf \op_x(n)$, the characterisation in terms of closed complexity cannot be improved to the same setting. In fact, it is already known that some aperiodic words can have $\liminf{\cl_x(n)}<+\infty$. 
For example, L.~Schaeffer and J.~Shallit showed in~\cite{SchSh} that $\liminf{\cl_x(n)}=0$ when $x$ is the paperfolding word. 
It is even possible for pure morphic words to have finite limit inferior for the closed complexity: For example, for the celebrated Cantor word $c$, also sometimes referred to as the Sierpinski word, one verifies that $\liminf{\cl_c(n)}=1$, and this value is attained for $n=7\cdot 3^k + 1$, for any $k$. 
The proof of this result can be easily obtained with case by case study, and thus is omitted; however it leads to the following question:

\begin{question}
For any $k\in \nats$, is it possible to find an aperiodic pure morphic word $x$ such that $\liminf{\cl_x(n)}=k$?
\end{question}

Although it is not possible to have the equivalence $\liminf{\cl_x(n)}$ bounded $\Longleftrightarrow x$ ultimately periodic, it might still be possible to obtain something stronger than our theorem: we already improved the first version of the theorem to the setting of syndetic sets, but it may be possible to get the same result in the case of piecewise syndetic sets. A set $S$ is piecewise syndetic if it is the intersection of a syndetic set and a thick set, hence if there are arbitrarily long intervals where it has bounded gaps.
\begin{question}
Is it true that, for any aperiodic word $x$ and any piecewise syndetic set $S$,\\ $\limsup\limits_{n\in S}{\cl_x(n)}=+\infty$?
\end{question}


\section*{Acknowledgements}
We would like to thank both Gabriele Fici and Luca Zamboni for suggesting this problem to us and for many useful discussions.

\bibliographystyle{abbrv}
\bibliography{ref}

\end{document}